\documentclass[12pt]{article}

\author{J.~F.~Feinstein\\School of Mathematical Sciences, University of Nottingham\\University Park, Nottingham, NG7 2RD, UK\\E-mail: Joel.Feinstein@nottingham.ac.uk\and M.~J.~Heath\footnote{This author is supported by post-doctoral grant SFRH/BPD/40762/2007 from FCT (Portugal) and completed most of this research while supported by a PhD grant from the EPSRC (UK).}\phantom{*}\footnote{This paper contains work from this author's PhD thesis \cite{MeThesis}.}\\Departamento de Matem\'atica, Instituto Superior T\'ecnico\\ Av.~Rovisco Pais, 1049-001 Lisboa, Portugal\\E-mail: mheath@math.ist.utl.pt}
\title{Swiss cheeses, rational approximation and universal plane curves.}

\date{}
\usepackage{amsfonts}
\usepackage{amsthm}
\usepackage{amsmath}
\usepackage{amssymb}
\usepackage[all]{xy}
\usepackage{bm}
\usepackage[pdftex]{graphicx}
\usepackage{color}

\newcommand{\C}{\mathbb C}
\newcommand{\R}{\mathbb R}
\newcommand{\N}{\mathbb N}

\newcommand{\eps}{\varepsilon}
\renewcommand{\-}{\setminus}
\newcommand{\norm}[1]{\left\Vert #1\right\Vert}

\newcommand{\abs}[1]{\left\vert #1 \right\vert}
\newcommand{\seq}[2]{(#1_{#2})_{#2\in\N}}

\newcommand{\func}[5]
{\begin{eqnarray*}
#1&:&#2\rightarrow#3,\\
&&#4\mapsto #5
\end{eqnarray*}}

\newcommand{\diam}{\mathrm{diam}}
\renewcommand{\d}{\,\mathrm d}
\newcommand{\Sier}{{S}ierpi\'nski }

\newcommand{\up}{\mathrm}

 \newcommand{\lopen}{\mathopen{]}}
 \newcommand{\ropen}{{\mathclose{[}}}
\newtheorem{thm}{Theorem}[section]
\newtheorem{cor}[thm]{Corollary}

\newtheorem{lem}[thm]{Lemma}
\newtheorem{prop}[thm]{Proposition}

\newtheorem{question}[thm]{Question}

\theoremstyle{definition}
\newtheorem{dfn}[thm]{Definition}
\theoremstyle{definition}

\theoremstyle{definition}

\newtheorem{ex}[thm]{Example}

\begin{document}
\baselineskip=17pt
\maketitle
\renewcommand{\thefootnote}{}



\renewcommand{\thefootnote}{\arabic{footnote}}

\setcounter{footnote}{0}
\begin{abstract}
In this paper we consider the compact plane sets known as \emph{Swiss cheese sets}, which are a useful source of  examples in the theory of uniform algebras and rational approximation. We develop a theory of \emph{allocation maps} connected to such sets and we use this theory to modify examples previously constructed in the literature to obtain examples homeomorphic to the \Sier carpet. Our techniques also allow us to avoid certain technical difficulties in the literature.
\end{abstract}
\section{Introduction and motivation}
In this paper we shall concern ourselves with ``Swiss cheese'' constructions. These represent a particular method for constructing compact subsets of the complex plane that has been used extensively in the theory of rational approximation and, more generally, in the theory of uniform algebras. In general little is specified about the topology of the sets produced by this technique. Since  uniform algebra theory has strong connections to topology, the topological properties of the sets on which we build our examples is an obvious thing to study. In this paper we shall show that it is possible to modify many Swiss-cheese-based examples related to uniform algebras and rational approximation, so that our compact plane set is homeomorphic to the well-known \Sier carpet.
\subsection{Basic uniform-algebraic concepts}
Throughout this paper by a \emph{compact space} we will mean a non-empty, compact, Hausdorff topological space. Let $X$ be a non-empty, locally compact, Hausdorff space. We denote the set of all continuous functions from $X$ to $\C$ which tend to zero at infinity by $C_0(X)$. If $X$ is a compact space this is equal to the set of all continuous $\C$-valued functions, which we denote by $C(X)$. Equipping $C(X)$ with the usual pointwise operations makes it a commutative, semisimple, complex algebra. If we further equip  $C(X)$ with the supremum norm $\norm\cdot_\infty$, it is standard that it is then a Banach algebra. We always treat $C(X)$ as a Banach algebra with this norm. We call a closed subalgebra $A$ of $C(X)$ a \emph{uniform algebra on $X$}
if it contains the constant functions and if, for all $x,y\in X$ with $x\ne y$, there is $f\in A$ with $f(x)\ne f(y)$. A uniform algebra on $X$ is \emph{trivial} if it is equal to $C(X)$ and \emph{non-trivial} otherwise.

A \emph{character} on a uniform algebra $A$ is a non-zero algebra homomorphism from $A$ into $\C$. A uniform algebra $A$ on a compact space $X$ is \emph{natural} if the only characters from $A$ into $\C$ are evaluations at points of $X$.

We shall use the term \emph{plane set} to mean ``subset of the complex plane''. For a non-empty, compact plane set $X$ we define
$R_0(X)$ to be the subalgebra of $C(X)$ consisting of functions $f=g|_X$ where $g:\C\rightarrow \C\cup\{\infty\}$ is a
rational function with $\infty\not\in g(X)$. We define $R(X)$ to be the supremum-norm closure of $R_0(X)$ in $C(X)$. It is standard that $R(X)$ is a natural uniform algebra on $X$.

\begin{dfn}
A uniform algebra, $A$, on a compact space, $X$,  is \emph{essential} if, for each closed, non-empty proper subset, $Y$, of $X$, there is a function
$f\in C(X)\-A$ such that $f|_Y=0$.
\end{dfn}
\begin{dfn}
 Let $X$ be a compact space, let $\mu$ be a regular Borel measure on $X$ and let $U\subset C(X)$. We say that $\mu$ is an
\emph{annihilating measure for} $U$ if $\int_X f \d\mu=0$ for all $f\in U$. We shall denote by $M(X)$ the Banach space of regular, complex
Borel measures on $X$ with the total variation norm.
\end{dfn}

 The following result is \cite[Theorem 2.8.1]{Browder} together with some observations made in the proof of that theorem.
\begin{prop}\label{essential}
Let $A$ be a uniform algebra on a compact space, $X$ and let $E(A)$ be the closure in $X$ of the union of the supports of all annihilating measures for $A$ on $X$.  Then $E(A)$ is the unique, minimal, closed subset of $X$ such that, for all $f\in C(X)$ with $f|_{E(A)}\subseteq\{0\}$, we have $f\in A$. Furthermore, $A|_{E(A)}$ is uniformly closed in $C(E(A))$, and
\[A=\{f\in C(X):f|_E\in A|_E\}.\]
The uniform algebra $A$ is essential if and only if $E(A)=X$, and $A=C(X)$ if and only if $E(A)=\emptyset$.
\end{prop}
We may think of the essential set of $A$ as being ``the set on which $A$ is non-trivial'' and ``essential'' as meaning ``everywhere
 non-trivial''.

\begin{dfn}
Let $A$ be a commutative Banach algebra and let $\psi$ be a character on $A$. A \emph{point derivation} at $\psi$ is a linear functional $d$ on $A$ such that
\[d(ab)=\psi(a)d(b)+\psi(b)d(a),\quad \textrm{for all }a,b\in A.\]
Let $n\in\N\cup\{\infty\}$. A \emph{point derivation of order} $n$ at $\phi$ is a sequence $(d_k)_{k=0}^n$ of linear functionals, such that $d_0=\phi$ and, for each $i<n+1$,
\[
d_i(fg)=\sum_{k=0}^id_k(f)d_{i-k}(g).
\]
We call $(d_k)_{k=0}^n$ \emph{bounded} if, for each $i<n+1$, $d_i$ is a bounded linear functional.
\end{dfn}

Let $A$ be a natural uniform algebra on a compact space $X$. We say that: $A$ is \emph{regular} if, for all $x\in X$ and all
compact sets $E\subseteq X\setminus\{x\}$, there exists $f\in A$ such that $f(E)\subseteq \{1\}$ and $f(x)= 0$; $A$ is \emph{normal} if, for every closed set $F\subseteq X$ and every compact set $E \subseteq X\setminus F$, there exists $f\in A$  such that $f(E)\subseteq \{1\}$ and $f(F)\subseteq \{0\}$. It is standard that regularity and normality are equivalent (see \cite[Proposition 4.1.18]{Dales}).
\subsection{Connections between uniform algebras and topology}
In order to motivate our results we shall discuss connections between the theory of uniform algebras and topology. The key observation is the following, which is basically trivial.
\begin{prop}
Let $\sf P$ be a property, which a Banach algebra may hold and which is invariant under Banach algebra isomorphism. Then, for a compact 
space $X$, ``there exists a uniform algebra on $X$ which satisfies $\sf P$'' and ``there exists a natural uniform algebra on $X$ which 
satisfies $\sf P$'' are topological properties of $X$.
\end{prop}

Thus, for a compact space $X$, it makes sense to consider questions of the form: ``Which Banach-algebraic properties may a (natural)
uniform algebra on $X$ have?'' These sorts of question have been little studied. In the many examples of uniform algebras constructed using Swiss-cheese techniques there is not typically any mention made of the topological properties of the underlying compact space.

Now, an obvious technique for constructing uniform algebras with different sets of properties on a fixed compact space is as follows. Let $Y$ be a compact space, let $X$ be a compact subspace of $Y$, and let $A$ be a uniform algebra on $X$.  We may define a uniform
algebra, $A^{(Y)}$, on $Y$ thus
\[A^{(Y)}=\{f\in C(Y):f|_X\in A\}.\]
Many properties of $A$ are then necessarily shared by $A^{(Y)}$. For example the following are easily proven and probably all
well-known (see, for example, \cite[Lemma 2.4.9]{MeThesis} for details).
\begin{lem}\label{subinj}
 Let $A$ be a uniform algebra on a compact space $X$, and let $Y$ be a compact space such that $X\subseteq Y$. Then:
\begin{itemize}
 \item[(a)] $A^{(Y)}$ is trivial if and only if $A$ is trivial;
\item[(b)] $A^{(Y)}$ is natural if and only if $A$ is natural;
\item[(c)] $A^{(Y)}$ is normal if and only if $A$ is normal;
\item [(d)]if $z\in X$, $n\in\N\cup\{\infty\}$ and $A$ has a non-zero bounded point derivation of order $n$ at $z$, then $A^{(Y)}$ has
a non-zero, bounded point derivation of order $n$ at $z$.
\end{itemize}
\end{lem}
Hence we may, for example, construct a non-trivial, natural, normal, uniform algebra on the compact unit disc. These example are somewhat artificial, since they are not essential.
\subsection{A survey of the use of Swiss cheese constructions in the theory of uniform algebras}
Examples in the theory of uniform algebras are often constructed by considering compact subsets of the complex plane obtained by
removing some sequence of open discs from a compact disc. Sets built in  such a way are usually called ``Swiss cheeses'' or ``Swiss
cheese sets''; we shall use the term ``Swiss cheese'' in a related but different sense, and we note that \emph{every} compact plane set 
may be constructed in this way. We let $X$ be a compact plane  set constructed by means of a Swiss cheese and  consider the uniform
algebra $R(X)$. By placing conditions on the radii and centres of the discs to be removed we are able to control certain
Banach-algebraic properties of $R(X)$.

For a (closed or open) disc $D$ in the plane, we let $r(D)$ be the radius of $D$. If $D$ is the empty set or a singleton we say
$r(D)=0$.
\begin{dfn}
 We shall call a pair, $\mathbf D=(\Delta, \mathcal D)\in\mathcal P(\C)\times \mathcal P(\mathcal{P}(\C))$, a \emph{Swiss cheese}
if $\Delta$ is a compact disc and $\mathcal D$ is a countable or finite collection of open discs. Let $\mathbf D=(\Delta, \mathcal D)$ be a Swiss cheese. We say that $\mathbf D$ is:
\emph{semiclassical} if  the discs in
$\mathcal D$ intersect neither one another, nor $\C\-\Delta$, if, for each $D\in\mathcal D$, $\overline D\subsetneq\Delta$ and $\sum_{D\in\mathcal D} r(D)< \infty$;
\emph{classical} if the closures of the discs in
$\mathcal D$ intersect neither one another nor $\C\-\up{int}\,{\Delta}$, and  $\sum_{D\in\mathcal D} r(D)< \infty$; \emph{finite} if $\mathcal D$ is finite.

Let $\mathbf D=(\Delta, \mathcal D)$ be a Swiss cheese. We call the plane set
$X_\mathbf D:=\Delta\-\bigcup\mathcal D$ the \emph{associated Swiss cheese set}. We say that a plane set $X$ is: a \emph{semiclassical 
Swiss cheese set} if there is a semiclassical Swiss
cheese $\mathbf D$ such that $X=X_\mathbf D$; a \emph{classical Swiss cheese set} if there is a classical Swiss
cheese $\mathbf D$ such that $X=X_\mathbf D$.
\end{dfn}
The earliest use of a Swiss cheese set in the theory of rational approximation  was in \cite{Roth}, where Roth constructed a classical Swiss cheese set $K$ with empty interior
such that $R(K)\ne C(K)$. This showed that there are compact plane sets such that $R(X)\ne A(X)$ where $A(K)$ is the uniform algebra of continuous functions on $K$ which are analytic
on the interior of $K$.   Roth's proof was essentially the same as that of the Theorem \ref{Swissess}, below.

A second example, showing how careful choice of the discs to be removed allows us to control the properties of $R(X)$ was given by
Steen in \cite{Steen}. This example is a classical Swiss cheese set $X$ such that $R(X)$ contains a non-constant, real valued function, 
something that had been conjectured to be impossible. Furthermore this function depended only on the real part of the independent
variable.

 We shall concentrate on Swiss cheeses
$\mathbf D$ such that the associated Swiss cheese set, $X_{\mathbf D}$, has empty interior in $\C$.
 We mention in passing
that Swiss cheese sets with non-empty interior are used, for example, in Examples 9.1, 9.2 and 9.3 of \cite{Gamelin}, to
demonstrate that a compact plane set $K$ may have dense interior and yet have $R(K)\ne A(K).$

We introduce some notation for integration over paths and chains. Further details may be found in Chapter 10 of \cite{Rudin}.
Let $(\gamma_1,\dots,\gamma_k)$ and $(\delta_1,\dots, \delta_n)$ be finite sequences of piecewise smooth paths in the plane. We say $(\gamma_1,\dots,\gamma_k)$ and $(\delta_1,\dots, \delta_n)$ are equivalent if, for all $f\in C_0(\C)$, we have
\[\sum_{i=1}^k\int_{\gamma_i}f \d z=\sum_{i=1}^n\int_{\delta_i}f \d z.\]
It is standard that this defines an equivalence relation on the set of all such sequence; we call the equivalence classes induced by
this relation \emph{chains}. We denote the chain containing $(\gamma_1,\dots, \gamma_k)$ by $\gamma_1\dotplus\dots\dotplus\gamma_k$.
Let $\Gamma=\gamma_1\dotplus\dots\dotplus\gamma_k$. We define integration over $\Gamma$ as follows:
\[\int_\Gamma f \d z:=\sum_{i=1}^k \int_{\gamma_i} f \d z\qquad(f\in C_0(\C)).\]

If $\gamma_1,\dots\gamma_n$
are chains we write $\gamma_1\dotplus\dots\dotplus\gamma_n$ for the chain with
\[\int_\Gamma f \d z=\int_{\gamma_1}f\d z+\dots+\int_{\gamma_n}f\d z\qquad (f\in C_0(\C)).\]
For a chain or piecewise smooth path $\gamma$ we define $\mu_\gamma$ to be the unique, regular, Borel measure on $\C$ satisfying
\[
\int_{\gamma^*}f\d\mu_\gamma=\int_{\gamma}f\d z\qquad (\textrm{for all }f\in C_0(\C)).
\]

\begin{thm}\label{Swissess}
Let $X$ be  a semiclassical Swiss cheese set. Then  $R(X)$ is essential.
\end{thm}
\begin{proof}
Suppose first that $z\in\up{int}(X)$ and let $r>0$ be sufficiently small that $\overline{B(z,\gamma)}\subseteq X$. We define a path $\gamma:[\pi,\pi]\rightarrow \C$ by $\gamma_{z,r}(t)=z+re^{it}$. Then, by
Cauchy's theorem, $\mu_{\gamma_{z,r}}$ is an annihilating measure for $R(X)$ and so, by Proposition \ref{essential},  $z\in E(R(X))$. We
shall show that there exists an annihilating measure, $\mu$, for $R(X)$ with $\up{supp}(\mu)=\partial X$ and so, by
Proposition \ref{essential}, $\partial X\subseteq E(R(X))$. Hence
we will have shown that $X=E(R(X))$. Let $\mathbf D=(\Delta, \mathcal D)$ be a semiclassical Swiss
cheese such that $X=X_\mathbf D$. We let $\gamma_\Delta$ be the boundary circle of $\Delta$ given the positive orientation.
For $D\in\mathcal D$, let  $\gamma_D$ be the boundary circle of $D$
given the negative orientation. Obviously $\up{supp}\left(\mu_{\gamma_\Delta}\right)=\partial\Delta$, $\up{supp}\left(\mu_{\gamma_D}\right)=\partial D$ $(D\in\mathcal D)$ and these measures are non-atomic.

Now, 
for each $D\in\mathcal D$, $\norm{\mu_{\gamma_D}}\le 2\pi r(D)$. Hence
\[
\sum_{D\in\mathcal D}\norm{\mu_{\gamma_D}} \le 2\pi\sum_{D\in\mathcal D}r(D)<\infty,
\]
so
\[\mu:=\mu_{\gamma_\Delta}+\sum_{D\in\mathcal D}\mu_{\gamma_D}\]
defines a measure $\mu\in M(X)$.
Clearly, if $Y$ is a closed subset of $\up{int}(X)$ then $\mu(Y)=0$, so $\up{supp}(\mu)\subseteq \partial X$. To show the reverse inequality, first note that for, $D\in\mathcal D\cup\{\Delta\}$ and $z\in X$,
$\mu_{\gamma_D}(\{z\})=0$. Now let $Y$ be a closed subset of $\partial\Delta$. Then $Y\cap\bigcup_{D\in\mathcal D}\overline D$ is countable and so $\mu(Y\cap\bigcup_{D\in\mathcal D}\overline D)=0$ and $\mu_\Delta(Y\cap\bigcup_{D\in\mathcal D}\overline D)=0$. Hence 
$\mu(Y)=\mu_\Delta(Y)$. Similarly, if we let $D\in \mathcal D$ and $Y\subseteq\partial D$ then
$\mu(Y)=\mu_D(Y)$. Hence, for each point $z\in \partial\Delta\cup\bigcup_{D\in\mathcal D}\partial D$ and every neighbourhood $U$ of $z$ there is a set $Y\subseteq U$ with $\mu(Y)\ne 0$. Thus,
$z\in\up{supp}(\mu)$ and so  $\partial\Delta\cup\bigcup_{D\in\mathcal D}\partial D\subseteq\up{supp}(\mu)$, but
$\overline{\partial\Delta\cup\bigcup _{D\in\mathcal D}\partial D}=\partial X$, so we have $\partial X\subseteq \up{supp}(\mu)$. Thus
$\partial X= \up{supp}(\mu)$. It only remains to show that $\mu$ is an annihilating measure for $R(X)$.
To show this we let $f\in R_0(X)$; then $f$ is holomorphic on the open set
\[V:=\C\-\{z\in\C: z\textrm{ is a pole of } f\}.\]
We shall assume that $\mathcal D$ is infinite; the proof in the case where $\mathcal D$ is finite is similar (and easier). We let $\seq{D}{n}$ be a sequence enumerating $\mathcal D$ and pick $N\in\N$ such that those poles of
$f$ which are contained in $\Delta$ all lie in $D_1\cup\dots\cup D_N$. Then
\[\Gamma_N=\gamma_{\Delta}\dotplus\gamma_{D_1}\dotplus\dots\dotplus\gamma_{D_N}\]
is a cycle with $\mathrm{Ind}(\Gamma, z)=0$ for all $z\in\C\-V$. Hence, by Cauchy's theorem,
\[\int_{\Gamma_N} f\d z=0,\]
and letting $N$ tend to infinity yields
\[\int f\d \mu=0.\]
Hence $\mu$ is an annihilating measure for $R(X)$.
\end{proof}

The first known example of a non-trivial uniform algebra with no non-zero bounded point derivations was due to Wermer,
\cite{Wermer}. In fact he proved the following.
\begin{prop}\label{prop} Let $\Delta$ be a closed disc in $\mathbb{C}$, and let $\varepsilon>0$. Then there is a classical
Swiss cheese $\mathbf D=(\Delta,\mathcal D)$ such that
\[\sum_{D\in\mathcal D} r(D)<\eps,\]
and $R(X_{\mathbf D})$ has no non-zero bounded point derivations.
\end{prop}

The first known example of a non-trivial, normal uniform algebra was due to McKissick,
\cite{McKissick}. In fact he proved the following.
\begin{prop}\label{regex}
 For any closed disc $\Delta$ and any $\eps>0$, there is a Swiss cheese, $\mathbf D =(\Delta,\mathcal D)$, such that,
\[
 \sum_{D\in\mathcal D} r(D)<\eps,
\]
and $R(X_{\mathbf D})$ is normal.
\end{prop}
This construction was
simplified somewhat by K\"{o}rner in \cite{Ko}. In \cite {O'Farrell} O'Farrell showed that in the above we could further insist that $0\in X_\mathbf D$ and that $R(X)$ have a bounded point
derivation of infinite order at $0$. The constructions of McKissick, K\"orner and O'Farrell appeared not to produce classical Swiss
cheese sets.
McKissick's result, along with what we would now call a system of Cole extensions (see, for example, \cite{Dawson}) was a crucial tool in Cole's (\cite{Cole}) solution to the famous ``peak point problem''.

The first author of this present paper has made use of Swiss cheese constructions to produce a variety of examples of plane sets $X$ such that the uniform algebra, $R(X)$ has interesting, specified properties. In \cite{FeinsteinStronglyRegular} he used McKissick's example, together with a system of Cole extensions to construct a non-trivial, strongly regular uniform algebra (see the paper for the definition). In \cite{FeinsteinTrivJen} the first author used a Swiss cheese  construction to obtain a compact plane set $X$ such that $R(X)$ has no non-trivial Jensen measures (see that paper for the definition) but is not normal. In \cite{FeinsteinMorris} the same author used a Swiss cheese construction to produce a counterexample to the conjecture (of Morris, in \cite{Morris}) that a uniform algebra with no non-zero, bounded point derivations would have to be weakly amenable (see \cite[Section 2.8]{Dales}). In \cite{me} the second author of the present paper showed that the uniform algebra produced could, in addition,  be normal, by using a Swiss-cheese-like method of removing discs from a compact 2-cell; more details  may be found in \cite{MeThesis}.
\subsubsection{An issue in the literature}\label{litprob}
Constructions using non-classical Swiss cheeses (including  the non-classical Swiss cheese constructions listed above) often rely upon the following result which appears on pages 28 and 29 of \cite{Bonsall}.
\begin{prop}\label{Bon}
 Let $\mathbf D=(\Delta, \{D_1, \dots, D_n\})$ be a finite Swiss cheese. Then $\partial X_{\mathbf D}$ consists of a
finite number of  arcs of the circles $\partial \Delta$ and $\partial D_i$. If we orient the arcs in
$\partial \Delta\cap\partial X_{\mathbf D}$ positively and those of each $\partial D_i\cap\partial X_{\mathbf D}$
negatively, this turns $\partial X_{\mathbf D}$ into a contour such that the following holds. If $D$ is an open neighbourhood of $X_{\mathbf D}$ and $f$ is analytic on $D$, then
\begin{enumerate}
 \item $\int_{\partial X_{\mathbf D}}f(z)\d z=0$;
\item $f(\zeta)=\frac{1}{2\pi i}
\int_{\partial X_{\mathbf D}}\frac{f(z)}{z-\zeta}\d z\qquad(\zeta\in\mathrm{int}(X_{\mathbf D})).$
\end{enumerate}
\end{prop}
However, the proof of this given in \cite{Bonsall} is a sketch, which appears somewhat difficult to make rigorous, and
we are not aware of any other proof of the result in print. This being the case, it may be helpful to have other methods available that do not depend on this result. Theorem
\ref{classical} of the current paper will provide an alternative means of proving the non-triviality of $R(X)$ for a large class of
non-classical Swiss cheese sets $X$.

We shall show that in many cases of Swiss cheeses $\mathbf D$ constructed so that $R(X_\mathbf D)$ has particular properties, we may assume that $\mathbf D$ is classical without losing the relevant properties. In the following subsection we show that this will mean that we have natural, essential uniform alebras with a variety of specified properties on a fixed compact space, namely the \Sier carpet.
\subsection{The \Sier carpet.}
The \Sier carpet is a well known fractal, which has been widely studied in topology, the theory of dynamical systems and complex analysis (see \cite{Bonk}). It is defined as follows.
We let $Q$ be the compact 2-cell (rectangle)  with corners at $0,1,i$ and $1+i$ and, for $z\in \C$ and $l\in \lopen 0,\infty\ropen$, we define $U(z,l)$ to be the open 2-cell with corners at $z, z+l, z+li$ and $z+l+li$. The \emph{\Sier carpet}, $S$,  is  the set,
\[
S=Q\-\bigcup_{k\in\N, \,m,n\in\left\{0,\dots3^{k-1}\right\}}U\left(3^{-k}((3m+1)+(3n+1)i),3^{-k}\right).
\]
Figure \ref{carp} shows an approximation of the \Sier carpet.
\begin{figure}
\begin{center}

\includegraphics[width=7cm]{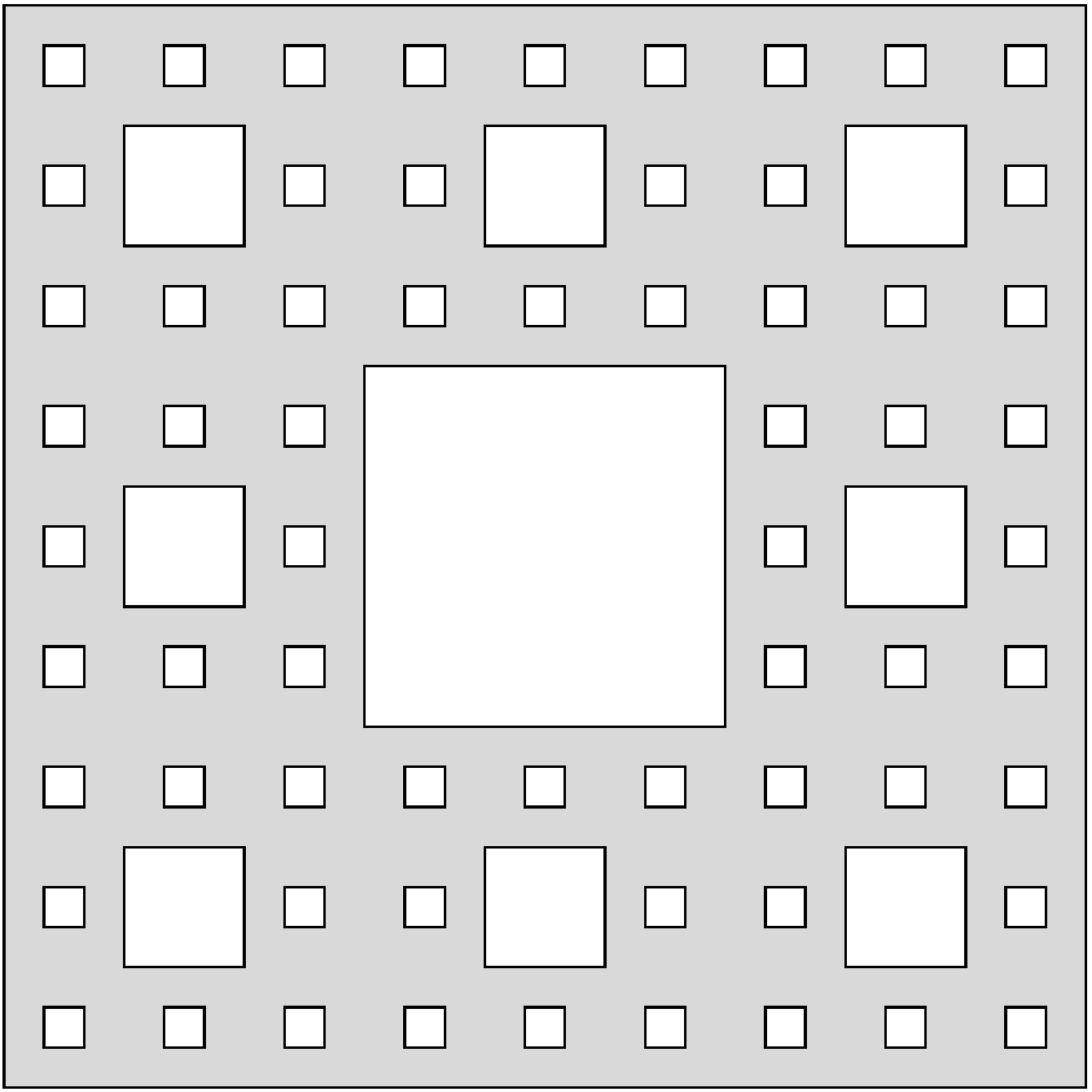}

\caption{The \Sier carpet.}
\label{carp}
\end{center}
\end{figure}

In this paper we consider how Swiss cheeses relate to plane homeomorphs of the \Sier carpet. Our first examples come as consequences of the following
result of Whyburn (\cite{Whyburn}), which may be found as \cite[Theorem 7.2]{Bonk}.
\begin{prop}\label{Sier}
Let $\Delta=\{z\in\C:\abs{z}\le 1\}$, let $\seq{D}{i}$ be a sequence
of pairwise disjoint Jordan domains whose closures lie in the interior of $\Delta$, and let
\[X:= \Delta\setminus\bigcup_{i\in\N} D_i.\] Then $X$ is homeomorphic to the
{S}ierpi\'nski carpet if and only if $X$ has empty interior, $\partial D_i \cap\partial D_j = \emptyset$ for $i \ne j$,
and $\diam(D_i) \rightarrow 0$ as $i \rightarrow \infty$.
\end{prop}
This gives us the following corollary.
\begin{cor}\label{Sierc}
 Let $X$ be a classical Swiss cheese set with empty interior in $\C$. Then $X$ is homeomorphic to the \Sier carpet.
\end{cor}
Thus, for any classical Swiss cheese set $X$ with empty interior we may consider $R(X)$ to be a uniform algebra on the \Sier
carpet. Each of these algebras is natural and, by Proposition
\ref{Swissess}, each is essential. We now consider how a well known topological property of the \Sier carpet relates to uniform algebras.
\begin{dfn}
Let $T$ be a non-empty topological space and $\mathcal U\in \mathcal P(\mathcal P(T))$ be an open cover of $T$. We say an open cover $\mathcal V$ is a
\emph{refinement} of $\mathcal U$ if, for each $V\in\mathcal V$, there exists $U\in\mathcal U$ with $V\subseteq U$. We define the
\emph{topological dimension} of $T$ to be the smallest non-negative integer $n$ (if it exists) such that every open cover  of $T$ has
a refinement $\mathcal V$ such that each $x\in T$ is in at most $n+1$ elements of $\mathcal V$. If no such integer exists, then we say the toplogical dimension is infinite.
\end{dfn}
For subsets of $\R^n$ the following, which is \cite[Theorem IV 3]{Hurewicz} holds.
\begin{prop}
 A subset $X$ of $\R^n$ has topological dimension strictly less than $n$ if and only if $X$ has empty interior in $\R^n$.
\end{prop}
\begin{dfn}
 A compact plane set $X$ is a \emph{universal plane curve} if it has topological dimension $1$, and whenever $Y$ is a compact plane set with
topological dimension less than or equal to $1$, then there is a subset $Y'$ of $X$ which is homeomorphic to $Y$.
\end{dfn}
The following was proven by \Sier in \cite{Sier} (see also \cite[p.433]{BlumMeng}).

\begin{prop}\label{topdim}
The \Sier carpet is a universal plane curve.
\end{prop}

The remainder of this paper deals with a technique for finding classical Swiss cheese sets (and thus homeomorphs of the \Sier carpet) as subsets plane sets that are built using Swiss cheeses, such as those discussed in our survey.
\section{Classicalisation of Swiss cheeses}
For a Swiss cheese $\mathbf D=(\Delta, \mathcal D)$, we define $\delta(\mathbf D)=r(\Delta)-\sum_{D\in\mathcal D}r(D)$. Note that $\delta (\mathbf D)<-\infty$ if and only if $\sum_{D\in\mathcal D}r(D)<\infty$.  We shall prove the following theorem.
\begin{thm}\label{classical}
For every Swiss cheese  $\mathbf D$ with $\delta(\mathbf D)>0$, there is a classical
Swiss cheese $\mathbf D'$ with $X_{\mathbf D'}\subseteq X_{\mathbf D}$ and $\delta(\mathbf D')\ge\delta(\mathbf D)$.
\end{thm}
Now, most of the examples mentioned in the survey section of this paper allow us to make a free choice of $\Delta$, and to specify that  $\sum_{D\in\mathcal D}r(D)$ be arbitrarily small. Hence, those plane sets may be taken to contain a classical Swiss cheese set as a subset. This is important because, if $X$ and $Y$ are compact plane sets with $Y\subseteq X$, then many properties of $R(X)$ are shared by $R(Y)$. We give some examples in the following proposition which is elementary, is probably well known, and appears as
\cite[Lemma 2.1.1]{MeThesis}.
\begin{prop}\label{subset}
 Let $X$ and $Y$ be compact plane sets with $Y\subseteq X$. Then:
\begin{itemize}
\item[(i)] if $R(X)$ is trivial then so is $R(Y)$;
 \item[(ii)] if $R(X)$ does not have any non-zero bounded point derivations, then neither does $R(Y)$;
\item[(iii)] if $R(X)$ is normal, then so is $R(Y)$.
\end{itemize}
\end{prop}

In order to prove Theorem \ref{classical} we shall need the following collection of facts. The proofs are elementary and may be found in \cite{MeThesis}.

\begin{prop}\label{Facts}
\begin{itemize}
\item[(a)] \label{discs} Let $\mathcal F$ be a non-empty, nested collection of open discs in $\C$, such that $\sup\{r(E):E\in\mathcal F\}<\infty$. Then
$\bigcup \mathcal F$ is an open disc, $E$, and there is a nested increasing  sequence
$\seq{D}{n}\subseteq \mathcal F$ such that $\bigcup_{n\in\N}D_n=E$. Furthermore, if we order $\mathcal F$  by inclusion,
\[r(E)=\lim_{n\rightarrow\infty}r(D_n)=\lim_{D\in\mathcal F}r(D)=\sup_{D\in\mathcal F}r(D).\]

\item [(b)] \label{discs2}
 Let $\mathcal F$ be a non-empty, nested collection of closed discs in $\C$. Then $\Delta:=\bigcap \mathcal F$ is a closed disc
or a singleton and there is a nested decreasing  sequence
$\seq{D}{n}\subseteq \mathcal F$ such that $\bigcap_{n\in\N}D_n=\Delta$. Furthermore, if we order $\mathcal F$ by reverse inclusion, then
\[r(\Delta)=\lim_{n\rightarrow\infty}r(D_n)=\lim_{D\in \mathcal F}r(D)=\inf_{D\in \mathcal F}r(D).\]
\end{itemize}
\end{prop}

\begin{dfn}\label{above}
Let $\mathbf D=(\Delta,\mathcal D)$ be a Swiss cheese. We define
\[\widetilde {\mathbf D}=\mathcal D\cup\{\C\-\Delta\}.\]
Now let $\mathbf E=(H,\mathcal E)$ be a second Swiss cheese, and let $f:\widetilde {\mathbf D}\rightarrow \widetilde {\mathbf E}.
$
We define $\mathcal G(f)=f^{-1}(\{\C\-H\})\cap\mathcal D$.
We say that $f$ is an \emph{allocation map} if  the following hold:
\begin{itemize}
\item[(A1)] for each $U\in\widetilde{\mathbf D}$, $U\subseteq f(U)$;
\item[(A2)]
$$\sum_{D\in \mathcal G(f)}r(D)\ge r(\Delta)-r(H);$$
\item[(A3)]  for each $E\in\mathcal E$,
$$\sum_{D\in f^{-1}(E)}r(D)\ge r(E).$$
\end{itemize}
If there is an allocation map from $\widetilde{\mathbf D}$ to $\widetilde{\mathbf E}$ we say that \emph{$\mathbf E$ is above
$\mathbf D$}.
\end{dfn}
Note that these axioms imply that $f$ is surjective. In particular, since there is no disc $D$ with $\C\-\Delta\subseteq D$ we have $f(\C\-\Delta)=\C\- H$.

Thus, if $\mathbf E$ is above $\mathbf D$, then (A1) implies that $H\subseteq\Delta$.
The following properties of of allocation maps are elementary consequences of the definition. Full details of the proofs may be found in \cite{MeThesis}.
\begin{prop}\label{allfacts}
\begin{itemize}
\item[(i)]\label{Rtran}
Let  $\mathbf D_1=(\Delta_1,\mathcal D_1)$, $\mathbf D_2=(\Delta_2,\mathcal D_2)$ and
$\mathbf D_3=(\Delta_3,\mathcal D_3)$ be Swiss cheeses
and let
\[
f:\widetilde{\mathbf D}_1\rightarrow\widetilde{\mathbf D}_{2}
\]
\[
g:\widetilde{\mathbf D}_2\rightarrow\widetilde{\mathbf D}_3
\]
be allocation maps. Then $g\circ f$,  is an allocation map form $\widetilde{\mathbf D_1}$ to $\widetilde{\mathbf D_3}$.
\item[(ii)]\label{Rref}
Let $\mathbf D=(\Delta,\mathcal D)$ be a Swiss cheese. Then the identity map from $\widetilde{\mathbf D}$ to itself is an allocation 
map. Suppose further that $\sum_{D\in\mathcal D}r(D)<\infty$. Then the identity map is the unique allocation map from
$\widetilde{\mathbf D}$ to itself.
\item[(iii)]\label{Ssub}
Suppose that $\mathbf D=(\Delta, \mathcal D)$ and $\mathbf E=(H, \mathcal E)$ are Swiss cheeses such that  $\mathbf E$ is above $\mathbf D$. Then $X_\mathbf E\subseteq X_\mathbf D$.
\item[(iv)]\label{abovelength}
Let $\mathbf D=(\Delta, \mathcal D)$ and $\mathbf E=(H, \mathcal E)$ be Swiss cheeses such that $\mathbf E$ is above $\mathbf D$.
Then
\[
\delta(\mathbf E)\ge \delta(\mathbf D).
\]
\end{itemize}
\end{prop}
We note that parts (i) and (ii) of the preceding proposition show that taking Swiss cheeses as objects, and allocation maps as morphisms gives a
(small) category. Thus, we may consider sub-categories such as the category of Swiss cheeses, $\mathbf D$, such that
$\delta(\mathbf D)>0$ and allocation maps.
Now fix a Swiss cheese $\mathbf D$, and let $\mathcal S(\mathbf D)$ be the collection of all pairs $(\mathbf E, f)$ such that
$\mathbf E$ is a Swiss cheese and $f:\widetilde{\mathbf D}\rightarrow\widetilde{\mathbf E}$ is an allocation map. Note that, for all $(\mathbf E, f)\in \mathcal S(D)$, $\mathbf E$ is above $\mathbf D$. We define a binary
relation, $\ge$, on $\mathcal S(\mathbf D)$ by saying $(\mathbf E', f')\ge(\mathbf E, f)$ if there is an allocation map
$g:\widetilde{\mathbf E}\rightarrow\widetilde{\mathbf E'}$ such that $g\circ f=f'$. Note that, since $f$ is onto, any such $g$ is unique.
\begin{lem}\label{allge}
Let $\mathbf D$ be a Swiss cheese such that $\delta(\mathbf D)>-\infty$. Then the binary relation $\ge$
defined above is a partial order on $\mathcal S(\mathbf D)$.
\end{lem}
\begin{proof}
First, we show that $\ge$ is reflexive. Let $(\mathbf E, f)\in\mathcal S(\mathbf D)$. By Proposition \ref{Rref} the identity map
$\up{id}:\widetilde{\mathbf E}\rightarrow\widetilde{\mathbf E}$ is an allocation map. Clearly $\up{id}\circ f=f$ and so
$(\mathbf E, f)\ge(\mathbf E, f)$.

Now, we show that $\ge$ is transitive. Let $(\mathbf E_1, f_1),(\mathbf E_2, f_2), (\mathbf E_3, f_3)\in\mathcal S(\mathbf D)$ such
that  $(\mathbf E_2, f_2)\ge(\mathbf E_1, f_1)$ and $(\mathbf E_3, f_3)\ge(\mathbf E_2, f_2)$. Then, there are allocation maps,
$g_{1,2}:\widetilde{\mathbf E}_1\rightarrow\widetilde{\mathbf E}_2$, such that $g_{1,2}\circ f_1=f_2$,  and
$g_{2,3}:\widetilde{\mathbf E}_2\rightarrow\widetilde{\mathbf E}_3$, such that $g_{2,3}\circ f_2=f_3$.
Set $g_{1,3}=g_{2,3}\circ g_{1,2}$. Then, by part (i) of Proposition \ref{Rtran}, $g_{1,3}$ is an allocation map from $\widetilde{\mathbf E_1}$ to $\widetilde{\mathbf E_3}$. Also
\[g_{1,3}\circ f_1=(g_{2,3}\circ g_{1,2})\circ f_1=g_{2,3}\circ (g_{1,2}\circ f_1)=g_{2,3}\circ f_2=f_3,\]
and so $(\mathbf E_3, f_3)\ge(\mathbf E_1, f_1)$.

Finally, we show that $\ge$ is antisymmetric. Let $(\mathbf E_1, f_1),(\mathbf E_2, f_2)\in\mathcal S(\mathbf D)$ such
that  $(\mathbf E_2, f_2)\ge(\mathbf E_1, f_1)$ and $(\mathbf E_1, f_1)\ge(\mathbf E_2, f_2)$. Then, there are allocation maps,
$g_{1,2}:\widetilde{\mathbf E}_1\rightarrow\widetilde{\mathbf E}_2$, such that $g_{1,2}\circ f_1=f_2$,  and
$g_{2,1}:\widetilde{\mathbf E}_2\rightarrow\widetilde{\mathbf E}_1$, such that $g_{2,1}\circ f_2=f_1$. Set
\[g=g_{2,1}\circ g_{1,2}:\widetilde{\mathbf E}_1\rightarrow\widetilde{\mathbf E}_1.\]
Then, by part (i) of Proposition \ref{Rtran}, $g$ is an allocation map. Since $\mathbf E$ is above $D$, $\delta(\mathbf E_1)>-\infty$ and so, by part (ii) of Proposition \ref{Rref}, $g$ is the identity map on $\widetilde{\mathbf E}_1$. Now, let
$U\in\widetilde{\mathbf E}_1$. It follows easily that
\[U\subseteq g_{1,2}(U)\subseteq g_{2,1}(g_{1,2}(U))=g(U)=U,\]
so $g_{1,2}(U)=U$. Similarly, if $U\in\widetilde{\mathbf E}_2$, then $g_{2,1}(U)=U$. Thus $(\mathbf E_1, f_1)=(\mathbf E_2, f_2)$.
\end{proof}
\begin{lem}\label{gemax}
Let $\mathbf D$ be a Swiss cheese such that $\delta(\mathbf D)>0$, and let $\mathcal C$ be a chain in
$(\mathcal S(\mathbf D),\ge)$. Then $\mathcal C$ has an upper bound in $(\mathcal S(\mathbf D),\ge)$.
\end{lem}
\begin{proof}
For $i\in\mathcal C$ we write $i=(\mathbf E_i, f_i)$ and $\mathbf E_i=(H_i,\mathcal E_i)$, and for $j\in\mathcal C$ with $j\ge i$, we let
\[g_{i,j}:\widetilde{\mathbf E}_i\rightarrow\widetilde{\mathbf E}_j,\]
be the unique (as discussed above) allocation map such that $g_{i,j}\circ f_i=f_j$. From uniqueness, it follows easily that \begin{equation}\label{compose}
 g_{i,k}=g_{j,k}\circ g_{i,j}\qquad(i\le j\le k\in\mathcal C).
\end{equation}

Note that $\{H_i:i\in\mathcal C\}$ is a nested decreasing collection of closed discs and, for each $D\in\mathcal D$, $\{f_i(D):i\in\mathcal C\}$ is a nested increasing collection of open plane sets.

By part (iv) of Proposition \ref{abovelength}, we have
\begin{equation}
r(H_i)\ge \delta (\mathbf E_i)\ge \delta(\mathbf D)>0.\label{rHi}
\end{equation}
Let $H=\bigcap_{i\in \mathcal C}H_i.$ By part (b) of Proposition \ref{discs2}, $H$ is a compact disc or singleton with
\[
r(H)=\lim_{i\in \mathcal C}r(H_i).
\]
By (\ref{rHi}) $r(H)>0$, so $H$ is a compact disc.
Now we define a map as follows
\func{f}{\widetilde{\mathbf D}}{\mathcal P(\C)}{U}{\bigcup_{i\in \mathcal C}f_i(U).}
First, we note that
\[
f\left(\C\-\Delta\right)=\bigcup_{i\in\mathcal C}\left(\C\-H_i\right)
=\C\-\bigcap_{i\in\mathcal C}H_i=\C\-H.
\]
Note also that, if $D\in\mathcal D$, then exactly one of the following two cases holds.
\begin{itemize}
\item[(i)] There exists $i\in\mathcal C$ such that $f_i(D)=\C\- H_i$. In this case, for $j\ge i$, we have, since $f_j=g_{i,j}\circ f_i$, that
$f_j(D)=g_{i,j}\left(\C\-H_i\right)=\C\-H_j$. Thus $f(D)=\C\-H$.
\item[(ii)] For each $i\in\mathcal C$, $f_i(D)\in\mathcal E_i$. In this case, $\{f_i(D):i\in\mathcal C\}$ is a
collection of open discs, with $f_i(D)\subseteq f_j(D)$, if $i\le j$. Also, for each $i\in \mathcal C$,
 we have (since $f_i$ satisfies (A3)), that
\begin{equation}
 r(f_i(D))\le\sum_{D\in f_i^{-1}(f_i(D))}r(D)\le \sum_{D\in \mathcal D}r(D)<\infty\label{rfi}.
\end{equation}
Thus, by part (a) of Proposition \ref{Facts}, $f(D)=\bigcup_{i\in\mathcal C}f_i(D)$ is an open disc with $r(f(D))=\lim_{i\in\mathcal C}r(f_i(D))$.
\end{itemize}
Hence,
\[
\mathcal E:=\left\{f(U):U\in \mathcal D\right\}\-\{\C\-H\}.
\]
is a collection of open discs.

Set $\mathbf E=(H,\mathcal E)$. By the above, $\mathbf E$ is a Swiss cheese. By definition, $f\left(\widetilde{\mathbf D}\right)=\widetilde{\mathbf E}$. We claim that $f$ (considered as map into
$\widetilde{\mathbf E}$) is an allocation map. That $f$ satisfies (A1) is trivial.

To show that $f$ satisfies (A2), note that, by the
argument for case (i), above,
\[ \bigcup_{i\in\mathcal C} \left(f_i^{-1}(\C\-H_i)\right)\subseteq f^{-1}(\C\-H),\]
i.e.
\[\mathcal G_i\subseteq\mathcal G(f)\quad(i\in\mathcal C).
\]
Thus, since $r(H)=\lim_{i\in \mathcal C}r(H_i)$, and each $f_i$ satisfies (A2),   we have
\begin{eqnarray}
\nonumber r(H)&=&\liminf_{i\in\mathcal C}r(H_i)\\
\nonumber &\ge&\liminf_{i\in\mathcal C}\left(r(\Delta)-\sum_{D\in \mathcal G(f_i)}r(D)\right)\\
&\ge&r(\Delta)-\sum_{D\in \mathcal G(f)}r(D)\label{a3}.
\end{eqnarray}
Hence, $f$ satisfies (A2).

To show that $f$ satisfies (A3), let $E\in\mathcal E$, and let $U\in\mathcal D$ be such that $f(U)=E$. Let
$i,j\in\mathcal C$ with $j\ge i$, and let $D\in\mathcal D$ such that $f_i(U)=f_i(D)$. Then, since $f_j=g_{i,j}\circ f_i$, we have that 
$f_j(D)=f_j(U)$, and so $f(D)=f(U)=E$.  Thus,
\[
\bigcup_{i\in \mathcal C}f_i^{-1}(f_i(U))\subseteq f^{-1}(f(U))=f^{-1}(E).
\]
Since $r(E)=\lim_{i\in\mathcal C}r(f_i(U))$, and each $f_i$ satisfies (A3), we have
\begin{eqnarray}
\nonumber r(E)&=&\limsup_{i\in \mathcal C}r(f_i(U))\\
\nonumber &\le&\limsup_{i\in \mathcal C}\left(\sum_{D\in f_i^{-1}(f_i(U))}r(D)\right)\\
\label{a4}&\le&\sum_{D\in f^{-1}(E)}r(D).
\end{eqnarray}
Thus $f$ satisfies (A3), and so is an allocation map.

We claim that $(\mathbf E, f)$ is the upper bound we require. To see this, let $i=(\mathbf E_i, f_i) \in\mathcal C$, define $I_i=\{j\in\mathcal C:j\ge i\}$, and take $U\in\widetilde{\mathbf E}_i$. Then, there exists $V\in \widetilde{\mathbf D}$ such that
 $U=f_i(V)$. Now let $j\in \mathcal C$ with  $ j\ge i$. Then  $g_{i,j}(U)=f_j(V)$. Thus we have
\[
\bigcup_{j\in I_i}g_{i,j}(U)=\bigcup_{j\in\mathcal C}f_{j}(V)=f(V)\in\widetilde{\mathbf E}.
\]
Hence, we can define a map,
\begin{eqnarray*}
g_i&:& \widetilde{\mathbf E}_i \rightarrow \widetilde{\mathbf E}\\
&& U \mapsto \bigcup_{j\in I_i}g_{i,j}(U),
\end{eqnarray*}
and we have $g_i\circ f_i=f$.

It remains to show that $g_i$ is an allocation map. To see this, note that it follows from the equation (\ref{compose}) that $I_i$ is a chain in $\mathcal S(\mathbf E_i)$, and so the
proof that $f$ is an allocation map also shows that $g_i$ is. The result follows.
\end{proof}
\begin{proof}[Proof of Theorem \ref{classical}]
By Proposition \ref{Rref} (part (ii)) and Lemmas \ref{allge} and \ref{gemax},
$\left(\mathcal S(\mathbf D), \ge\right)$ is a non-empty, partially ordered set such that every chain has an upper bound. Hence, we may apply Zorn's lemma to obtain a maximal element $(\mathbf E, f)$ of
$\left(\mathcal S(\mathbf D), \ge\right)$. By part (iii) of  Proposition \ref{Ssub}, we have that
$X_\mathbf E\subseteq X_\mathbf D$. Since $\mathbf E$ is above $\mathbf D$, $\delta (\mathbf E)\ge\delta(\mathbf D)>0$.

It remains to show that $\mathbf E$ is a classical Swiss cheese. Towards a contradiction, we assume otherwise.
 Then we must have at least one of the
following cases.
\begin{itemize}
\item[Case 1:] There exist $E, E'\in\mathcal E$ such that $\overline E\cap \overline E'\ne \emptyset$. In this case there exists an open disc $E''$ with $E\cup E'\subseteq E''$ and $r(E'')\le r(E)+r(E')$, as in Figure \ref{amal1}.
\begin{figure}
\begin{center}

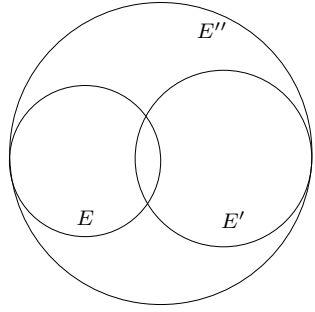

\caption{$E$, $E'$ and $E''$.}
\label{amal1}
\end{center}
\end{figure}
Let
$\mathcal E'=\left(\mathcal E\-\{E, E'\}\right)\cup\{E''\}$ and $\mathbf E'=(H,\mathcal E')$
and define $g:\widetilde{\mathbf E}\rightarrow\widetilde{\mathbf E'}$ by
$$g(U)=\left\{
\begin{array}{ll}E''&\textrm{if }U\in\{E,E'\}\\
U&\textrm{otherwise.}
\end{array}\right.$$
Then it is easy to check that $g$ is an allocation map. By part (i) of Proposition \ref{Rtran},  $g\circ f$ is an allocation map, and so
$(\mathbf E', f\circ g)\in\mathcal S(\mathbf D)$ with $(\mathbf E', f\circ g)>(\mathbf E, f)$.
\item[Case 2:] There exists $E\in\mathcal E$ such that $\overline E\not\subseteq \up{int}(H)$. Assume that this case holds and Case 1 does not.

 By the condition on the sum of the radii, $\up{int}(H)\not\subseteq \overline E$.  Then there exists a compact disc $H'$ such that  $D\subseteq \left(\C\-H'\right)$ and $r(H')\ge r(H)-r(E)$ as in Figure \ref{amal2}.
\begin{figure}
\begin{center}

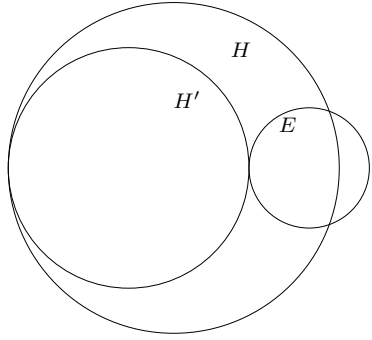

\caption{$E$, $H$ and $H'$.}
\label{amal2}
\end{center}
\end{figure}
 Let
$\mathcal E'=\mathcal E\-\{E\}$ and $\mathbf E'=(H',\mathcal E')$ and define
$g:\widetilde{\mathbf E}\rightarrow\widetilde{\mathbf E'}$ by
\[
g(U)=\left\{
\begin{array}{ll}\C\-H'&\textrm{if }U\in\{E,\C\-H\}\\
U&\textrm{otherwise.}
 \end{array}\right.
 \]
Then it is easy to check that $g$ is an allocation map. By part (i) of Proposition \ref{Rtran} $g\circ f$ is an allocation map and so 
$(\mathbf E', f\circ g)\in\mathcal S(\mathbf D)$ with $(\mathbf E', f\circ g)>(\mathbf E, f)$.
\end{itemize}
In either case we have a contradiction to the maximality of $(\mathbf E, f)$. The result follows.
\end{proof}
We are grateful to Prof.~J.~K.~Langley for pointing out to to us that the method we use to combine discs in Case 1 has previously appeared in the literature in the setting of finite unions of open discs. Zhang implicitly uses this method on page 50 of \cite{Zhang}.

\smallskip

Theorem \ref{classical} has the following  purely topological corollary.
\begin{cor}
 Let $\mathbf D$ be a Swiss cheese such that $\delta (\mathcal D)>0$ and
$X_{\mathbf D}$ has empty interior in $\C$. Then  $X_{\mathbf D}$ is a universal plane curve.
\end{cor}
\begin{proof}
By Theorem \ref{classical} there is a classical Swiss cheese set $Y$, with  $Y\subseteq X_{\mathbf D}$. By Corollary \ref{Sierc} $Y$ is 
homeomorphic to the \Sier carpet $S$. Let $E$ be a compact plane set with topological dimension less than or equal to $1$. Then by  Theorem \ref{topdim} there is  a plane set $E'$ homeomorphic to $E$ with $E'\subseteq Y\subseteq X_\mathbf D$.
\end{proof}
Note, in particular, that for any two such Swiss cheese sets,   each may be  continuously embedded in the other.

We are now able to use known examples of non-classical Swiss cheeses $X$ such that $R(X)$ has particular properties to construct new
examples using classical Swiss cheeses (in particular, to produce examples of essential uniform algebras on the \Sier carpet). We give 
the following example.
\begin{ex}There is a classical Swiss cheese set $X$ such that  $R(X)$ is normal.
\end{ex}
\begin{proof}
 By Proposition \ref{regex},  there is a Swiss cheese $\mathbf D =(\Delta,\mathcal D)$,  such that
\[\delta(\mathbf D)>0,\]
and $R(X_{\mathbf D})$ is normal. By Theorem \ref{classical}, there is a classical Swiss cheese set $X$ with $X\subseteq X_\mathbf D$. By Proposition \ref{subset}, $R(X)$ is normal.
\end{proof}
We note that, by Proposition \ref{prop} and Proposition \ref{subset}, we could in addition insist that $R(X)$ have no non-zero bounded point derivations.

We do not yet know whether the techniques in this paper can be adapted so that they preserve the existence of point derivations, or of other derivations into the dual of $R(X)$. Proofs of the existence of such derivations  often use Proposition \ref{Bon}  (see, for example, \cite{FeinsteinMorris}). If one wishes to avoid using that result, other ``work-arounds'' can typically be found (see, for example, Theorem 3.3.8 of \cite{MeThesis}).
\section{Open questions}
We finish with some open questions.
\begin{question}\label{sierq1}
Let $X$ be a compact plane set such that $R(X)\ne C(X)$. Does it follow that $X$ has a subset $S$ homeomorphic to the \Sier carpet?
\end{question}
\begin{question}\label{sierq2}
Let $X$ be a compact plane set such that $R(X)\ne C(X)$. Does it follow that $X$ has a subset $S$ homeomorphic to the \Sier carpet such that $R(S)\ne C(X)$?
\end{question}
\begin{question}\label{sierq2.5}
Let $X$ be a plane set such that $R(X)\ne C(X)$. Does it follow that $X$ has a subset which is a classical Swiss cheese set?
\end{question}
\begin{question}\label{sierq3}
Let $X$ be a compact plane set such that there exists a non-trivial, natural uniform algebra on $X$. Does it follow that $X$ has a subset  homeomorphic to the \Sier carpet?
\end{question}
\begin{question}\label{sierq4}
Let $X$ be a compact metric space such that there exists a non-trivial, natural uniform algebra on $X$. Does it follow that $X$ has a subset homeomorphic to the \Sier carpet?
\end{question}
We note that a positive answer to Question \ref{sierq3} would imply a negative answer to the following - a famous problem due to Gel$'$fand.
\begin{question}\label{Gelf}
Is there a non-trivial, natural uniform algebra on th interval $[0,1]$?
\end{question}

\bibliographystyle{amsplain}

\begin{thebibliography}{10}

\bibitem{BlumMeng}
L.~M. Blumenthal and K.~Menger, \emph{Studies in geometry}, W. H. Freeman and
  Co., San Francisco, Calif., 1970. \MR{MR0273492 (42 \#8370)}

\bibitem{Bonk}
Mario Bonk, \emph{Quasiconformal geometry of fractals}, International Congress
  of Mathematicians. Vol. II, Eur. Math. Soc., Z\"urich, 2006, pp.~1349--1373.
  \MR{MR2275649}

\bibitem{Bonsall}
F.~F. Bonsall and J.~Duncan, \emph{Complete normed algebras}, Springer-Verlag,
  New York, 1973, Ergebnisse der Mathematik und ihrer Grenzgebiete, Band 80.
  \MR{MR0423029 (54 \#11013)}

\bibitem{Browder}
A.~Browder, \emph{Introduction to function algebras}, W. A. Benjamin, Inc., New
  York-Amsterdam, 1969. \MR{MR0246125 (39 \#7431)}

\bibitem{Cole}
B.~J. Cole, \emph{One point parts and the peak point conjecture}, Ph.D. thesis,
  Yale University, 1968.

\bibitem{Dales}
H.~G. Dales, \emph{Banach algebras and automatic continuity}, London
  Mathematical Society Monographs. New Series, vol.~24, The Clarendon Press
  Oxford University Press, New York, 2000, Oxford Science Publications.
  \MR{MR1816726 (2002e:46001)}

\bibitem{Dawson}
T.~Dawson, \emph{A survey of algebraic extensions of commutative, unital normed
  algebras}, Function spaces (Edwardsville, IL, 2002), Contemp. Math., vol.
  328, Amer. Math. Soc., Providence, RI, 2003, pp.~157--170. \MR{MR1990397
  (2004e:46059)}

\bibitem{FeinsteinStronglyRegular}
J.~F. Feinstein, \emph{A nontrivial, strongly regular uniform algebra}, J.
  London Math. Soc. (2) \textbf{45} (1992), no.~2, 288--300. \MR{MR1171556
  (93i:46086)}

\bibitem{FeinsteinTrivJen}
\bysame, \emph{Trivial {J}ensen measures without regularity}, Studia Math.
  \textbf{148} (2001), no.~1, 67--74. \MR{MR1881440 (2002k:46127)}

\bibitem{FeinsteinMorris}
\bysame, \emph{A counterexample to a conjecture of {S}. {E}. {M}orris}, Proc.
  Amer. Math. Soc. \textbf{132} (2004), no.~8, 2389--2397 (electronic).
  \MR{MR2052417 (2005f:46098)}

\bibitem{Gamelin}
T.~W. Gamelin, \emph{Uniform algebras}, Prentice-Hall Inc., Englewood Cliffs,
  N. J., 1969. \MR{MR0410387 (53 \#14137)}

\bibitem{me}
M.~J. Heath, \emph{A note on a construction of {J}. {F}.\ {F}einstein}, Studia
  Math. \textbf{169} (2005), no.~1, 63--70. \MR{MR2139642 (2005m:46083)}

\bibitem{MeThesis}
\bysame, \emph{Bounded derivations from {B}anach algebras}, Ph.D. thesis,
  University of Nottingham, 2008.

\bibitem{Hurewicz}
W.~Hurewicz and H.~Wallman, \emph{Dimension {T}heory}, Princeton Mathematical
  Series, v. 4, Princeton University Press, Princeton, N. J., 1941.
  \MR{MR0006493 (3,312b)}

\bibitem{Ko}
T.~W. K{\"o}rner, \emph{A cheaper {S}wiss cheese}, Studia Math. \textbf{83}
  (1986), no.~1, 33--36. \MR{MR829896 (87f:46090)}

\bibitem{McKissick}
R.~McKissick, \emph{A nontrivial normal sup norm algebra}, Bull. Amer. Math.
  Soc. \textbf{69} (1963), 391--395. \MR{MR0146646 (26 \#4166)}

\bibitem{Morris}
S.~E. Morris, \emph{Bounded derivations from uniform algebras}, Ph.D. thesis,
  University of Cambridge, 1993.

\bibitem{O'Farrell}
A.~G. O'Farrell, \emph{A regular uniform algebra with a continuous point
  derivation of infinite order}, Bull. London Math. Soc. \textbf{11} (1979),
  no.~1, 41--44. \MR{MR535795 (81i:46062)}

\bibitem{Roth}
A.~Roth, \emph{Approximationseigenschaften und {S}trahlengrenzwertemeromorpher
  und ganzer {F}unktionen}, Comment. Math. Helv. \textbf{11} (1938), 77--125.

\bibitem{Rudin}
W.~Rudin, \emph{Real and complex analysis}, McGraw-Hill, Singapore, 1987.

\bibitem{Sier}
W.~Sierpi{\'n}ski, \emph{On curves which contain the image of any given
  curve.}, Oeuvres choisies, PWN--\'Editions Scientifiques de Pologne, Warsaw,
  1975, p.~780 pp. (1 plate). \MR{MR0414303 (54 \#2406)}

\bibitem{Steen}
Lynn~A. Steen, \emph{On uniform approximation by rational functions}, Proc.
  Amer. Math. Soc. \textbf{17} (1966), 1007--1011. \MR{MR0199416 (33 \#7561)}

\bibitem{Wermer}
J.~Wermer, \emph{Bounded point derivations on certain {B}anach algebras}, J.
  Functional Analysis \textbf{1} (1967), 28--36. \MR{MR0215105 (35 \#5948)}

\bibitem{Whyburn}
G.~T. Whyburn, \emph{Topological characterization of the {S}ierpi\'nski curve},
  Fund. Math. \textbf{45} (1958), 320--324. \MR{MR0099638 (20 \#6077)}

\bibitem{Zhang}
Zhang Guan-Hou, \emph{Theory of entire and meromorphic functions,
  deficient and asymptotic values and singular directions}, Translations of
  Math Monographs, no. 122, American Mathematical Society, Providence, 1993,
  Chinese name order: family first.

\end{thebibliography}

\def\cprime{$'$}
\providecommand{\bysame}{\leavevmode\hbox to3em{\hrulefill}\thinspace}
\providecommand{\MR}{\relax\ifhmode\unskip\space\fi MR }
\providecommand{\MRhref}[2]{%
  \href{http://www.ams.org/mathscinet-getitem?mr=#1}{#2}
}
\providecommand{\href}[2]{#2}

\end{document}